\documentclass[12pt]{article}
\usepackage{amsmath,amssymb,graphicx,amsthm}

\newtheorem{theorem}{Theorem}[section]
\newtheorem*{theorem:repeat}{\tref{butterflystab}}
\newtheorem*{theorem:repeatmain}{\tref{main}}
\newtheorem{lemma}[theorem]{Lemma}
\newtheorem{corollary}[theorem]{Corollary}

\newtheorem{prop}[theorem]{Proposition}
\newtheorem{conjecture}[theorem]{Conjecture}

\newcommand\lref[1]{Lemma~\ref{lem:#1}}
\newcommand\tref[1]{Theorem~\ref{thm:#1}}
\newcommand\cref[1]{Corollary~\ref{cor:#1}}

\newcommand\sref[1]{Section~\ref{sec:#1}}

\newcommand\cjref[1]{Conjecture~\ref{conj:#1}}

\newcommand\bC{{\mathbf C}}
\newcommand\cA{{\mathcal A}}

\newcommand\cC{{\mathcal C}}

\newcommand\cF{{\mathcal F}}
\newcommand\cG{{\mathcal G}}

\newcommand\cP{{\mathcal P}}

\newcommand\cS{{\mathcal S}}

\begin{document}

\title{Induced and non-induced forbidden subposet problems}

\author{Bal\'azs Patk\'os\thanks{Research supported by the J\'anos Bolyai Research Scholarship of the Hungarian Academy of Sciences.} \\
\small MTA--ELTE Geometric and Algebraic Combinatorics Research Group, \\
\small H--1117 Budapest, P\'azm\'any P.\ s\'et\'any 1/C, Hungary \\
\small Alfr\'ed R\'enyi Institute of Mathematics, Hungarian Academy of Sciences, Hungary. \\
\small  \tt  patkosb@cs.elte.hu and patkos@renyi.hu.}

\maketitle

\begin{abstract}
The problem of determining the maximum size $La(n,P)$ that a $P$-free subposet of the Boolean lattice $B_n$ can have, attracted the attention of many researchers, but little is known about the induced version of these problems. In this paper we determine the asymptotic behavior of $La^*(n,P)$, the maximum size that an induced $P$-free subposet of the Boolean lattice $B_n$ can have for the case when $P$ is the complete two-level poset $K_{r,t}$ or the complete multi-level poset $K_{r,s_1,\dots,s_j,t}$ when all $s_i$'s either equal 4 or are large enough and satisfy an extra condition.
We also show lower and upper bounds for the non-induced problem in the case when $P$ is the complete three-level poset $K_{r,s,t}$. These bounds determine the asymptotics of $La(n,K_{r,s,t})$ for some values of $s$ independently of the values of $r$ and $t$.
\end{abstract}

\section{Introduction}
We use standard notation: $2^X$ denotes the power set of $X$, $\binom{X}{k}$ denotes the set of $k$-element subsets of $X$, for two sets $A\subset B$ the interval $\{G: A\subseteq G\subseteq B\}$ is denoted by $[A,B]$ and $[n]$ stands for the set of the first $n$ positive integers $\{1,2,\dots,n\}$. The complement $[n]\setminus A$ of a subset $A$ of $[n]$ will be denoted by $\overline{A}$ and for a family $\cF\subseteq 2^{[n]}$ of sets we will write $\overline{\cF}=\{\overline{F}:F\in\cF\}$.

\vskip0.2truecm

The very first theorem in extremal finite set theory is due to Sperner \cite{S} and it states that if $\cF \subseteq 2^{[n]}$ is a family of sets that does not contain two sets $F_1,F_2$ with $F_1 \subsetneq F_2$, then $|\cF| \le \binom{n}{\lfloor \frac{n}{2}\rfloor}$ holds. Such families are called \textit{antichains} or \textit{Sperner families}. A first generalization is due to Erd\H os \cite{E}, who proved that if $\cF$ does not contain any $(k+1)$-chains, i.e., $k+1$ sets $F_1,F_2, \dots, F_{k+1}$ with $F_1 \subsetneq F_2 \subsetneq \dots \subsetneq F_{k+1}$, then $|\cF|\le \Sigma(n,k):=\sum_{i=1}^k\binom{n}{\lfloor\frac{n-k}{2}\rfloor+i}$ holds. Such families are called \textit{$k$-Sperner families}.

These two theorems have many applications and generalizations. One such generalization is the topic of forbidden subposet problems first introduced by Katona and Tarj\'an \cite{KT}. We say that a poset $Q$ \textit{contains} another poset $P$ if there is an injection $i:P\rightarrow Q$ such that for every $p_1,p_2\in P$ the fact $p_1 \le p_2$ implies $i(p_1) \le i(p_2)$. If $Q$ does not contain $P$, then it is said to be \textit{$P$-free}. If $\cP$ is a set of posets, then $Q$ is $\cP$-free if it is $P$-free for all $P\in\cP$. The parameter introduced by Katona and Tarj\'an is the quantity $La(n,P)$ that denotes the maximum size of a $P$-free subposet of $B_n$, the Boolean poset of all subsets of $[n]$ ordered by inclusion. With this notation Erd\H os's theorem states that $La(n,P_{k+1})=\Sigma(n,k)$, where $P_{k+1}$ denotes the path on $k+1$ elements, i.e., a total ordering on $k+1$ elements.

In the same paper, Katona and Tarj\'an introduced the induced version of the problem. We say that $Q$ \textit{contains an induced copy of $P$} if there is an injection $i:P\rightarrow Q$ such that for any $p_1,p_2 \in P$ we have $p_1 \le p_2$ if and only if $i(p_1) \le i(p_2)$. If $Q$ does not contain an induced copy of $P$, then $Q$ is said to be \textit{induced $P$-free}. The analogous extremal number is denoted by $La^*(n,P)$ and obviously the inequality $La(n,P) \le La^*(n,P)$ holds for any poset $P$. The notation for multiple forbidden subposets is $La(n,\cP)$ and $La^*(n,\cP)$.

As any poset $P$ is contained in $P_{|P|}$, we clearly have $La(n,P)\le La(n,P_{|P|})=\Sigma(n,|P|-1).$ Strengthenings of this general bound were obtained by Burcsi and Nagy \cite{BN_doublechain}, Chen and Li \cite{CL_general} and recently by Gr\'osz, Methuku and Tompkins \cite{GMT}. Therefore it is natural to compare $La(n,P)$ to $\binom{n}{\lfloor\frac{n}{2}\rfloor}$. Unfortunately, it is not known whether $\pi(P)=\lim_{n\rightarrow \infty}\frac{La(n,P)}{\binom{n}{\lfloor\frac{n}{2}\rfloor}}$ exists. The following conjecture was first stated in \cite{GL_general}.

\begin{conjecture}
\label{conj:big} For any poset $P$ let $e(P)$ denote the largest integer $k$ such that for any $j$ and $n$ the family $\cup_{i=1}^k\binom{[n]}{j+i}$ is $P$-free. Then $\pi(P)$ exists and is equal to $e(P)$.
\end{conjecture}

This conjecture has been verified for many classes of posets. The most remarkable result is due to Bukh.

\begin{theorem}
\label{thm:bukh} Let $T$ be a tree poset. Then $\Sigma(n,h(T)-1)\le La(n,T) \le (h(T)-1+O(\frac{1}{n}))\binom{n}{\lfloor \frac{n}{2}\rfloor}$ holds. In particular, $\pi(T)=e(T)$ holds for any tree poset $T$.
\end{theorem}

Much less is known about the induced version of the problem. It has only been proved recently by Methuku and P\'alv\"olgyi \cite{MP} that for every poset $P$ there exists a constant $c_P$ such that $La^*(n,P) \le c_P\binom{n}{\lfloor \frac{n}{2}\rfloor}$ holds. (For a special class of posets this has already been established by Lu and Milans \cite{LM}.) As the list of known results on forbidden induced subposet problems is very short here we enumerate all such theorems.

\begin{theorem}[Katona, Tarj\'an \cite{KT}]
\label{thm:nonind}
For $n\ge 3$ we have $La(n,\{\wedge,\vee\})=La^*(n,\{\wedge,\vee\})=2\binom{n-1}{\lfloor n/2\rfloor}$.
\end{theorem}

\begin{theorem}[Katona, Tarj\'an \cite{KT} and Carroll, Katona \cite{CK}]
\label{thm:carkat}
$(1+\frac{1}{n}+O(\frac{1}{n^2}))\binom{n}{\lfloor n/2\rfloor}\le La(n, \vee)=La(n,\wedge)\le La^*(n, \vee)=La^*(n,\wedge)\le (1+\frac{2}{n}+O(\frac{1}{n^2}))\binom{n}{\lfloor n/2\rfloor}$.
\end{theorem}

Finally, the induced version of \tref{bukh} has been proved, but only with an $o(1)$ error term instead of $O(\frac{1}{n})$.

\begin{theorem}[Boehnlein, Jiang \cite{BJ}]
\label{thm:bj} Let $T$ be a tree poset. Then $\Sigma(n,h(T)-1)\le La^*(n,T) \le (h(T)-1+o(1))\binom{n}{\lfloor \frac{n}{2}\rfloor}$ holds.
\end{theorem}

Before we state our results, let us formulate the induced analogue of \cjref{big}.

\begin{conjecture}
Let $P$ be a poset and let $e^*(P)$ denote the largest integer $k$ such that for any $j$ and $n$ the family $\cup_{i=1}^k\binom{[n]}{j+i}$ is induced $P$-free. Then $\pi^*(P)=\lim_{n\rightarrow \infty}\frac{La^*(n,P)}{\binom{n}{\lfloor \frac{n}{2}\rfloor}}$ exists and is equal to $e^*(P)$.
\end{conjecture}

In the present paper, we address both the induced and the non-induced problem for complete multi-level posets. Let $K_{r_1,r_2,\dots, r_s}$ denote the poset on $\sum_{i=1}^sr_i$ elements $a^1_1,a^1_2,\dots, a^1_{r_1}$, $a^2_1, a^2_2,\dots, a^2_{r_2},\dots, a^s_1,a^s_2,\dots, a^s_{r_s}$ with $a^i_{h} < a^j_{l}$ if and only if $i<j$. The \textit{rank} $r(a^i_l)$ of the element $a^i_l$ is $i$. Our first result gives not only the asymptotics of $La^*(n,K_{r,t})$, but also the order of magnitude of the second order term of the extremal value. The constructions that show the lower bounds in this and later theorems are based on the same idea that will be described at the beginning of \sref{pr}.

\begin{theorem}
\label{thm:twopart}
For any positive integers $2 \le r,t$ we have $\Sigma(n,2)+(\frac{r+t-2}{n}-O_{r,t}(\frac{1}{n^2}))\binom{n}{\lfloor n/2\rfloor}\le La^*(n,K_{r,t})\le (2+\frac{2(r+t-2)}{n}+o(\frac{1}{n}))\binom{n}{\lfloor n/2\rfloor}$.
\end{theorem}

Note that the same upper bound for $La(n,K_{r,t})$ follows from \tref{bukh} as $K_{r,t}$ is an (induced) subposet of $K_{r,1,s}$ and $K_{r,1,s}$ is a tree poset. By the same argument, \tref{bj} implies the asymptotics of $La^*(n,K_{r,t})$ but its error term is worse than that of \tref{twopart}. Let us remark that $La(n,K_{2,2})=\Sigma(n,2)$ was shown by De Bonis, Katona, Swanepoel \cite{DKS}. As they also showed the uniqueness of the extremal family, it was known that the strict inequality $La(n,K_{2,2})<La^*(n,K_{2,2})$ holds. \tref{twopart} tells us the order of magnitude of the gap between these two parameters.

Then we turn our attention to the three level case of $K_{r,s,t}$. To do so we need to introduce the following notation: for positive integers $r,t$ let 

\begin{eqnarray*}
f(r,t)=\left\{
\begin{array}{cc} 
0 & \textnormal{if}\ ~ r=t=1,\\
1 & \textnormal{if}\ r=1, t>1 \hskip 0.2truecm\textnormal{or}\ r>1,t=1, \\
2 & \textnormal{if}\ r,t> 2.
\end{array}
\right.
\end{eqnarray*}

\newpage

Also, for any integer $s\ge 2$ let us define $m=m_s= \lceil\log_2(s-f(r,t) + 2)\rceil$ and $m^*=m^*_s=\min\{m: s\le  \binom{m}{\lceil m/2\rceil}\}$ and for any real number $z$, let $z^+$ denote $\max\{0,z\}$. Note that $m^*_s$ is the minimum integer $m$ such that $2^{[m]}$ contains an antichain of size $s$ and thus an interval $[A,B]$ contains an antichain of size $s$ if and only if $|B\setminus A|\ge m^*_s$. Another equivalent formulation is to say that an interval $[A,B]$ contains an induced copy of $K_{1,s,1}$ if and only if $|B\setminus A|\ge m^*_s$. Similarly, an interval $[A,B]$ contains a non-induced copy of $K_{1,s,1}$ if and only if $|B\setminus A|\ge \lceil\log_2(s-f(1,1) + 2)\rceil$. It may seem foolish to denote 0 by $f(1,1)$, but we will see later how the function $f$ comes into the picture.

\vskip 0.3truecm

Our next theorem deals with the non-induced problem for complete three-level posets $K_{r,s,t}$. The main term of all of our bounds depends on the value of $r$ and $t$ via the function $f$. For most values of $s$ we can determine $\pi(K_{r,s,t})$, for the rest we obtain an upper bound that is bigger than our lower bound by less than one.

\begin{theorem}
\label{thm:threepartnonind}
Let $s-f(r,t) \ge 2$. 

(1) If $s-f(r,t) \in [2^{m_s-1} -1,2^{m_s} - \binom{m_s}{\lceil \frac{m_s}{2}\rceil}-1]$, then  $\pi(K_{r,s,t}) = e(K_{r,s,t}) = m_s+f(r,t)$ holds. In particular, we have 

$\Sigma(n,m_s+f(r,t))+\left(\frac{(r-2)^++(t-2)^+}{n}-O_{r,t}(\frac{1}{n^2})\right)\binom{n}{\lceil\frac{n}{2}\rceil}\le La(n,K_{r,s,t}) \le (m_s+f(r,t)+\frac{2(r+t-2)}{n}+o(\frac{1}{n}))\binom{n}{\lceil \frac{n}{2}\rceil}.$

\vskip 0.2truecm

(2) If $s-f(r,t) \in [2^{m_s} - \binom{m_s}{\lceil \frac{m_s}{2}\rceil},2^{m_s} -2]$, then

$\Sigma(n,m_s+f(r,t))+(\frac{(r-2)^++(t-2)^+}{n}-O_{r,t}(\frac{1}{n^2}))\binom{n}{\lceil\frac{n}{2}\rceil} \le La(n,K_{r,s,t})\le (m_s+f(r,t)+1- \frac{2^{m_s}-s+f(r,t)-1}{\binom{m}{\lceil \frac{m_s}{2}\rceil}})\binom{n}{\lceil\frac{n}{2}\rceil}$ holds.
\end{theorem}

Note that the special case $r=t=1$ of \tref{threepartnonind} was already obtained by Griggs, Li and Lu \cite{GLL_dia}.
Let us state a result that covers the case $s=2$, $f(r,t)>0$.

\begin{theorem}
\label{thm:rest} For any pair of integers $r,t$ with $f(r,t)>0$ we have $\Sigma(n,3)+(\frac{(r-2)^++(t-2)^+}{n}-O_{r,t}(\frac{1}{n^2}))\binom{n}{\lceil \frac{n}{2}\rceil}\le La(n,K_{r,2,t}) \le (3+\frac{2(r+t-2)}{n}+o(\frac{1}{n}))\binom{n}{\lceil \frac{n}{2}\rceil}.$ In particular, $\pi(K_{r,2,t})=3$ holds.
\end{theorem}

As we mentioned earlier Burcsi and Nagy \cite{BN_doublechain} obtained the general bound $$\limsup_n\frac{La(n,P)}{\binom{n}{\lceil n/2\rceil}}\le b(P):=\frac{|P|+L(P)}{2}-1,$$ where $L(P)$ is the length of the longest chain in $P$. Consequently, they proved \cjref{big} whenever $e(P)=b(P)$ holds. They provided seven small such examples all of which are complete $1$-, $2$-, or 
$3$-level posets. Furthermore, they introduced two operations how to obtain new posets $P$ satisfying $e(P)=b(P)$ starting from two other posets $P_1,P_2$ possessing this property. 
All resulting posets are complete multilevel ones. In particular, they obtained $La(n,K_{1,2,2})=La(n,K_{2,2,1})=\Sigma(n,3)$.

We turn our attention to the general case of $K_{r,s_1,s_2,\dots, s_j,t}$. As there are more technical details in calculating $e(K_{r,s_1,s_2,\dots, s_j,t})$ than in calculating $e^*(K_{r,s_1,s_2,\dots, s_j,t})$ we will only consider the induced problem in its full generality.

\begin{prop}
\label{prop:ind} 

(i) If $s_i\ge 2$ holds for all $1 \le i \le j$, then we have $e^*(K_{r,s_1,s_2,\dots, s_j,t})=f(r,t)+\sum_{i=1}^jm^*_{s_i}$.

(ii) Let us write $w=|\{i:s_{i-1}=s_i=1\}|$, where $r=s_0$ and $t=s_{j+1}$. Then $e^*(K_{r,s_1,s_2,\dots, s_j,t})=w+e^*(K_{r,\sigma_1,\sigma_2,\dots, \sigma_{j'},t})$, where $\sigma_1,\sigma_2,\dots, \sigma_{j'}$ is the sequence obtained from $s_1,s_2,\dots, s_j$ by removing all its ones.
\end{prop}

\begin{proof}
To see (i), let $\cF$ consist of $f(r,t)+\sum_{i=1}^jm^*_{s_i}$ consecutive levels of $2^{[n]}$ and suppose we find an induced copy of $K_{r,s_1,s_2,\dots, s_j,t}$. If $F_1,\dots, F_r$ and $F'_1,\dots,F'_t$ play the role of the bottom $r$ and the top $t$ sets, then $|\cap_{i=1}^tF'_i|-|\cup_{k=1}^rF_j|<\sum_{l=1}^jm^*_{s_i}$ holds. If $F^{i}_1,\dots,F^{i}_{s_{i}}$ play the role of the sets of the $i$th middle level of $K_{r,s_1,s_2,\dots, s_j,t}$, then $|\cup_{j=1}^{s_i}F^{i}_j|\ge |\cup_{j=1}^{s_{i-1}}F^{i-1}_j|+s_j$ must hold. Thus one would need $\sum_{i=1}^j m^*_{s_i}$ more levels for the $j$ middle levels of $K_{r,s_1,s_2,\dots, s_j,t}$. It is easy to see that $f(r,t)+\sum_{i=1}^jm^*_{s_i}+1$ consecutive levels do contain an induced copy of $K_{r,s_1,s_2,\dots, s_j,t}$.

To see (ii), assume $\cG$ is a copy of an induced $K_{r,s_1,s_2,\dots, s_j,t}$ in $2^{[n]}$. Let $i,i+p$ be two indices such that $s_i,s_{i+p+1}\ge 2$ and $s_{i+h}=1$ for all $1\le h\le p$. Let $G^{i}_1,\dots, G^{i}_{s_i}$ and $G^{i+p+1}_1,\dots, G^{i+p+1}_{s_{i+p+1}}$ denote the sets in $\cG$ corresponding to the $i$th and $(i+p+1)$st level of $K_{r,s_1,s_2,\dots, s_j,t}$. Then for $I=\cup_{l=1}^{s_i}G^{i}_l$ and $J=\cap_{l=1}^{s_{i+p+1}}G^{i+p+1}_l$ we must have $I \subseteq J$ and $|J|-|I|\ge p-1$ as $\cG$ contains a chain of length $p$ in $[I,J]$. For $G \in \cG$ let us write $r(G)$ for the rank of the element corresponding to $G$. Then $\cG'=\{G\in \cG:r(G)\le i\} \cup \{G\setminus (J\setminus I): G\in \cG, r(G)\ge i+p+1\}$ is an induced copy of $K_{r,s_1,s_2,\dots,s_i,s_{i+p+1},\dots s_j,t} $ such that the size of the largest set in $\cG'$ is $(p-1)$ less than than the size of the largest set in $\cG$. Continuing this process we obtain a copy of $K_{r,\sigma_1,\sigma_2,\dots, \sigma_{j'},t}$ where the size of the largest set is $w$ less than the size of the largest set in $\cG$. This shows $e^*(K_{r,\sigma_1,\sigma_2,\dots, \sigma_{j'},t})\le e^*(K_{r,s_1,s_2,\dots, s_j,t})-w$. To see the other inequality, one has to reverse the above procedure. We leave the details to the reader.
\end{proof}

\begin{theorem}
\label{thm:threepartind} 

(i) For any positive integers $1 \le r,t$ we have $\Sigma(n,4+f(r,t))+(\frac{r+t-2}{n}-O_{r,t}(\frac{1}{n^2}))\binom{n}{\lceil n/2\rceil}\le La^*(n,K_{r,4,t})\le (4+f(r,t)+\frac{2(r+t-2)}{n}+o(\frac{1}{n}))\binom{n}{\lfloor n/2\rfloor}$. 

\noindent In particular, $\pi^*(K_{r,4,t})=4+f(r,t)$ holds.

(ii) For any constant $c$ with $1/2<c<1$ there exists an integer $s_c$ such that if $s \ge s_c$ and $s\le c\binom{m^*_s}{\lceil m^*_s/2\rceil}$, then we have
 $\Sigma(n,m^*_s+f(r,t))+(\frac{r+t-2}{n}-O_{r,t}(\frac{1}{n^2}))\binom{n}{\lceil n/2\rceil}\le La^*(n,K_{r,s,t})\le(m^*_s+f(r,t)+\frac{2(r+t-2)}{n}+o(\frac{1}{n}))\binom{n}{\lfloor n/2\rfloor}$. In particular, $\pi^*(K_{r,s,t})=m^*_s+f(r,t)$ holds.

(iii) There exists an integer $s_0$ such that for any $r,s,t$ with $s\ge s_0$ we have $\Sigma(n,m^*_s+f(r,t))+(\frac{r+t-2}{n}-O_{r,t}(\frac{1}{n^2}))\binom{n}{\lceil n/2\rceil}\le La^*(n,K_{r,s,t})\le (m^*_s+1+f(r,t)+\frac{2(r+t-2)}{n}+o(\frac{1}{n}))\binom{n}{\lfloor n/2\rfloor}$.

(iv) For any constant $c$ with $1/2<c<1$ there exists an integer $s_c$ such that if all $s_i$'s satisfy that either $s_i=4$ or $s_i \ge s_c$ and $s_i\le c\binom{m^*_{s_i}}{\lceil m^*_{s_i}/2\rceil}$, then we have $La^*(n,K_{r,s_1,s_2,\dots, s_j,t})=(e^*(K_{r,s_1,s_2,\dots, s_j,t})+O_{r,t}(\frac{1}{n}))\binom{n}{\lfloor n/2\rfloor}$.
\end{theorem}

\vskip 0.2truecm

Our main technique to prove all four theorems is the chain partition method \cite{GLL_dia, GL_partition}. The remainder of the paper is organized as follows: in \sref{lemmas} we prove some preliminary lemmas and introduce all the necessary definitions that will be used in the proofs of \tref{twopart}, \tref{threepartnonind}, \tref{rest}, and \tref{threepartind}. Then in \sref{pr} we prove our results.

\section{Preliminary definitions and lemmas}
\label{sec:lemmas}

In this section we prove some preliminary lemmas that will serve as building blocks in the proofs of our main theorems. Before stating and proving these lemmas, let us enumerate all definitions that we will use later on in the paper.

Let $\bC_n$ denote the set of maximal chains in $[n]$.
For a family $\cF\subseteq 2^{[n]}$ of sets and $A \subseteq [n]$  we define $s^-_{\cF}(A)$ to be the maximum size of an antichain in $\{ F \in\cF:F \subseteq A\}$. and $s^+_{\cF}(A)$ to be the maximum size of an antichain in $\{ F \in\cF:A \subseteq F\}$. For a set $A\subseteq [n]$ and a family $\cF$ of sets let $\bC_{A,k,-}$ denote the set of those maximal chains $\cC$ from $\emptyset$ to $A$ for which for every $C \in \cC\setminus \{A\}$ we have $s^-_{\cF}(C)<k$ and let $\bC_{A,k,+}$ denote the set of those maximal chains $\cC$ from $A$ to $[n]$ for which for every $C \in \cC\setminus \{A\}$ we have $s^+_{\cF}(C)<k$.

The \textit{min-max-partition} of $\bC_n$ (introduced by Griggs Li and Liu in \cite{GLL_dia}) with respect to a family $\cF\subseteq 2^{[n]}$ is $\{\bC_{A,B}: A\subseteq B \subseteq [n]\}$ where $\bC_{A,B}$ consists of those maximal chains in $\bC_n$ of which the smallest set that belongs to $\cF$ is $A$ and of which the largest set that belongs to $\cF$ is $B$. To obtain a real partition of $\bC_n$ one has to add $\bC_{\emptyset}=\{\cC\in \bC_n: \cC\cap \cF=\emptyset\}$.

For $r\ge 2$ let us now define the \textit{$min_r$-partition} of $\bC_n$ with respect to $\cF$. For a set $A$ with $s^-_{\cF}(A)\ge r$ we set $\bC_{A,r}=\{\cC\in\bC_n: A \in \cC, \forall C\subset A, C\in \cC: s^-_{\cF}(C)<r\}$. Note that every $\cC\in\bC_n$ belongs to exactly one set $\bC_{A,r}$ provided $\cF$ contains an antichain of size $r$ as then $s^-_{\cF}([n])\ge r$ and $[n]$ is contained in all maximal chains $\cC\in\bC_n$. Thus $\{\bC_{A,r}: s^-_{\cF}(A)\ge r\}$ is a partition of $\bC_n$.

Now we define the \textit{$min_r-max_t$-partition} of $\bC_n$. Before introducing the formal definition, we describe the idea of the partition. For the sake of simplicity assume that both $r$ and $t$ are at least $2$. For every chain $\cC \in \bC_n$ we want to introduce two markers $A,B\in \cC$ with the property that $A$ is the smallest set in $\cC$ below which there exists an antichain of size $r$ in $\cF$ (i.e., $s^-_{\cF}(A)\ge r$) and $B$ is the largest set in $\cC$ above which there exists an antichain of size $t$ in $\cF$ (i.e., $s^+_{\cF}(B)\ge t$). If $\cF$ is $K_{r,s,t}$-free, we know that $[A,B]\cap \cF$ contains less than $s$ sets, while if $\cF$ is induced $K_{r,s,t}$-free, then $[A,B]$ does not contain an antichain of size $s$. The problem with the above reasoning is that $B\subsetneq A$ might hold, thus we will have to distinguish two cases.

Let us start with introducing $\cS=\{S\in 2^{[n]}: s_{\cF}^-(S)\ge r\}$, the family of those sets that can play the role of $A$ in the above argument. We partition $\cS$ into two subfamilies: $\cS^-=\{S \in \cS: s^+_{\cF}(S)<t\}$ and $\cS^+=\cS\setminus \cS^-$. Clearly, if $A\in \cS^-$ is the smallest set in the chain $\cC\in\bC_n$ with $s^-_{\cF}(A)\ge r$, then for the largest set $B$ in $\cC$ with $s^+_{\cF}(B)\ge t$ we will have $B\subsetneq A$. 

For any set $S \in\cS^-$ let $\bC_S$ denote the set of those maximal chains $\cC$ in $\bC_n$ in which
\begin{itemize}
\item
if $r=1$, then $S$ is the smallest set in $\cF \cap \cC$,
\item
if $r\ge 2$, then $S$ is the smallest set in $\cC$ with $s^-_{\cF}(S)\ge r$.
\end{itemize}
For any set $A \in S^+$ and $B$ with $A \subseteq B$ let $\bC_{A,B}=\bC_{A,r,B,t}$ denote the set of those maximal chains $\cC$ in $\bC_n$ in which
\begin{itemize}
\item
if $r=1$, then $A$ is the smallest set in $\cF \cap \cC$,
\item
if $r\ge 2$, then $A$ is the smallest set in $\cC$ with $s^-_{\cF}(A)\ge r$,
\item
if $t=1$, then $B$ is the largest set in $\cF \cap \cC$,
\item
if $t\ge 2$, then $B$ is the largest set in $\cC$ with $s^+_{\cF}(B)\ge t$.
\end{itemize}
The $\min_r-\max_t$-partition of $\bC_n$ is $\{\bC_S: S\in \cS^-\}\cup \{\bC_{A,B}: A\in \cS^+,A\subseteq B\}$. Consider a maximal chain $\cC\in\bC_n$. If $r\ge 2$ and the size $z$ of the largest antichain in $\cF$ satisfies $z=s^-_{\cF}([n])\ge \max\{r,t\}$, then there is a smallest set $H$ of $\cC$ with $s^-_{\cF}(H)\ge r$. If $H \in \cS^-$, then $\cC$ belongs to $\bC_H$. If not, then $H \in \cS^+$ and thus for the largest set $H'$ of $\cC$ with $s^+_{\cF}\ge t$ we have $H \subseteq H'$ and therefore $\cC \in \bC_{H,H'}$ holds. We obtained that the $\min_r-\max_t$-partition of $\bC_n$ is indeed a partition if $r\ge 2$. If $r=1$, then we need to add the set $\bC_{\emptyset}=\{\cC\in\bC_n:\cC\cap \cF=\emptyset\}$.

\vskip 0.2truecm

After introducing the necessary definitions, we start to prove our preliminary lemmas that will serve as building blocks of our proofs in \sref{pr}.

\vskip 0.2truecm

\begin{lemma}
\label{lem:noincomp} Let $\cF\subseteq 2^{[n]}$ be a family such that all $F\in \cF$ have size in $[n/2-n^{2/3},n/2+n^{2/3}]$. 

(i) Let $A\subset [n]$ with $s^-_{\cF}(A)<k$. Then the number of pairs $(F,\cC)$ where $\cC$ is a maximal chain from $\emptyset$ to $A$ and $F\in \cF \cap (\cC\setminus \{A\})$ is $\frac{2(k-1)}{n}|A|!+o(\frac{1}{n}|A|!)$.

(ii) Let $A\subset [n]$ with $s^+_{\cF}(A)<k$. Then the number of pairs $(F,\cC)$ where $\cC$ is a maximal chain from $A$ to $[n]$ and $F\in \cF \cap (\cC\setminus \{A\})$ is $\frac{2(k-1)}{n}(n-|A|)!+o(\frac{1}{n}(n-|A|)!)$.
\end{lemma}

\begin{proof}
We start by proving (i). The property possessed by $A$ and $\cF$ ensures that $\cF_A:=\{F\in \cF: F\subset A\}$ contains at most $k-1$ sets of each possible size. Thus the number of pairs $(F,\cC)$ in question is at most 
\[
\sum_{i=n/2-n^{2/3}}^{\min\{n/2+n^{2/3},|A|-1\}}(k-1)i!(|A|-i)!\le \frac{k-1}{|A|}|A|!+\frac{2(k-1)}{|A|(|A|-1)}|A|!+\frac{12(k-1)n^{2/3}}{|A|(|A|-1)(|A|-2)}|A|!
\]
\[
\le \frac{2(k-1)}{n}|A|!+O_k(\frac{1}{n^{4/3}}|A|!)
\]
 if $|A|\ge (1/2-n^{-1/3})n$. If $|A|\le (1/2-n^{-1/3})n=n/2-n^{2/3}$, then $\cF$ does not contain any subset $F$ of $A$. This completes the proof of (i) and (ii) follows by applying (i) to the set $\overline{A}$ and the family $\overline{\cF}$.
\end{proof}

\noindent \textbf{Remark.} Note that $n^{2/3}$ could be replaced by any function $f(n)$ satisfying $4\log n\sqrt{n}\le f(n)=o(n)$. In the proof of \lref{noincomp} we used $f(n)=o(n)$ and at the beginning of the proofs of upper bounds in \sref{pr}, we will need a calculation involving Chernoff's inequality where the assumption $4\log n\sqrt{n}\le f(n)$ will be used.

\begin{corollary}
\label{cor:plus} Let $\cF\subseteq 2^{[n]}$ be a family such that all $F\in \cF$ have size in $[n/2-n^{2/3},n/2+n^{2/3}]$. 

(i) Let $A\subset [n]$ with $s^-_{\cF}(A)\ge k$. Then the number of pairs $(F,\cC)$ where $\cC\in \bC_{A,k,-}$ and $F\in \cF \cap (\cC\setminus \{A\})$ is $(1+\frac{2(k-1)}{n})|\bC_{A,k,-}|+o(\frac{1}{n}|\bC_{A,k,-}|)$.

(ii) Let $A\subset [n]$ with $s^+_{\cF}(A)\ge k$. Then the number of pairs $(F,\cC)$ where $\cC\in\bC_{A,k,+}$ and $F\in \cF \cap (\cC\setminus \{A\})$ is $(1+\frac{2(k-1)}{n})|\bC_{A,k,+}|+o(\frac{1}{n}|\bC_{A,k,+}|)$.
\end{corollary}

\begin{proof}
First we prove (i). Let $A_1,\dots ,A_j,A_{j+1},\dots,A_{|A|}$ denote the subsets of $A$ of size $|A|-1$ such that $s^-_{\cF}(A_i)<k$ if and only if $1\le i \le j$. (If $s^-_{\cF}(A_i)\ge k$ for all $i$, then $\bC_{A,k,-}$ is empty and there is nothing to prove.) Note that if $S_1\subset S_2$, then $s^-_{\cF}(S_2)< k$ implies $s^-_{\cF}(S_1)<k$. Therefore $\bC_{A,k,-}=\cup_{i=1}^j\bC_{A_i,A}$, where
$\bC_{A_i,A}$ denotes the set of those maximal chains from $\emptyset$ to $A$ that contain $A_i$. Indeed, $\bC_{A_i,A} \subset \bC_{A,k,-}$ for $1\le i\le j$ as by the above $A$ is the smallest set in a chain $\cC \in \bC_{A_i,A}$ with $s^-_{\cF}(A)$ at least $k$, while for all $i\ge j+1$ we have $s^-_{\cF}(A_j)\ge k$ and thus $\bC_{A_j,A}\cap \bC_{A,k,-}=\emptyset$.

Let us fix $i$ with $1\le i\le j$ and consider pairs $(F,\cC)$ with $F\in\cF\cap \cC$ and $\cC\in\bC_{A_i,A}$. As $s^-_{\cF}(A_i)<k$, we can apply \lref{noincomp} (i) to $\cF$ and $A_i$, and obtain that the number of such pairs with $F \subsetneq A_i$ is at most $\frac{2(k-1)}{n}|A_i|!+o(\frac{1}{n}|A_i|!)$. Even if all $A_i$'s belong to $\cF$, then every chain $\cC\in\bC_{A,k,-}$ can contain one more set from $\cF$, namely one of the $A_i$'s. This completes the proof of (i) and (ii) follows by applying (i) to the set $\overline{A}$ and the family $\overline{\cF}$.
\end{proof}

\begin{lemma}
\label{lem:ind4} 

(i) Let $\cG\subseteq 2^{[n]}$ be a family of sets such that any antichain $\cA \subset \cG$ has size at most 3. Then the number of pairs $(G,\cC)$ with $G \in \cG \cap \cC$ and $\cC \in \bC_n$ is at most $4n!$.

(ii) For any constant $c$ with $1/2<c<1$ there exists an integer $s_c$ such that if $s \ge s_c$ and $s\le c\binom{m^*_s}{\lceil m^*_s/2\rceil}$, then the following holds: if $\cG\subseteq 2^{[n]}$ is a family of sets such that any antichain $\cA \subset \cG$ has size less than $s$, then the number of pairs $(G,\cC)$ with $G \in \cG \cap \cC$ and $\cC \in \bC_n$ is at most $m^*_sn!$.

(iii) There exists an integer $s_0$ such that if $s\ge s_0$ and $\cG\subseteq 2^{[n]}$ is a family of sets such that any antichain $\cA \subset \cG$ has size at most $s$, then the number of pairs $(G,\cC)$ with $G \in \cG \cap \cC$ and $\cC \in \bC_k$ is at most $(m^*_s+1)n!$.
\end{lemma}

\begin{proof}
First we prove (i). We may assume that $\emptyset, [n] \in \cG$ holds as adding them will not result in violating the condition of the lemma and the number of pairs to be counted can only increase. These two sets are in $n!$ maximal chains each, thus giving $2n!$ pairs. Any other set $G$ belongs to $|G|!(n-|G|)!=\frac{n!}{\binom{n}{|G|}}$ chains in $\bC_n$. Sets of same size form an antichain, therefore for every $1 \le i \le n-1$ there exist at most 3 sets of size $i$ in $\cG$ and thus the total number of pairs $(G,\cC)$ is at most
\[
S(n)=2n!+3n!\sum_{i=1}^{n-1}\frac{1}{\binom{n}{i}}.
\]
For $n=2,3,4,5$ the sum $S(n)$ equals $3.5n!,4n!,4n!, 3.8n!$, respectively. Furthermore, it is an easy exercise to show that $\frac{S(n)}{n!}$ is monotone decreasing for $n\ge 5$ and therefore $\frac{S(n)}{n!} \le 4$ holds for all positive integers $n$. This completes the proof of (i).

\vskip 0.2truecm

Now we prove (ii). Clearly, as long as $n < m^*_s$ we can have $\cG=2^{[n]}$ and then the number of pairs is $(n+1)n!\le m^*_sn!$. When $n\ge m^*_s$ we again use the observation that for any $0 \le j\le n$ we have $|\{G\in \cG\cap \binom{[n]}{j}|<s$ and thus the number of pairs $(G,\cC)$ is at most $S(n)=\sum_{j=0}^n\min\{s-1,\binom{n}{j}\}j!(n-j)!$.
We need to show that $R(n):=\frac{S(n)}{n!}=\sum_{j=0}^n\min\{\frac{s-1}{\binom{n}{j}},1\}\le m^*_s$ holds for all $n\ge m^*_s$. Consider the case $n=m^*_s$. If $s$ is large enough (and thus $m^*_s$ and $n$), then $\binom{m^*_s}{\lceil m^*_s/2\rceil}=(1+o(1))\binom{m^*_s}{\lceil m^*_s/2\rceil +j}$ holds provided $|j|\le \sqrt{m^*_s}/\log m^*_s$. Therefore, by the assumption $s\le c\binom{m^*_s}{\lceil m^*_s/2\rceil}$ we have at least $2\sqrt{m^*_s}/\log m^*_s$ summands in $R(m^*_s)$ that are not more than $\frac{1+c}{2}$, a constant smaller than 1. Thus, if $m^*_s$ is large enough, their subsum 
$$\sum_{i=\lceil m^*_s/2\rceil-\sqrt{m^*_s}/\log m^*_s}^{\lceil m^*_s/2\rceil+\sqrt{m^*_s}/\log m^*_s}\frac{s-1}{\binom{m^*_s}{j}}$$ 
is less than $2\sqrt{m^*_s}/\log m^*_s-1$ and since all other summands are not more than 1, we obtain $R(m^*_s)<m^*_s$.

To finish the proof of (ii), we prove that if $n \ge m^*_s$ holds, then we have $R(n+1) \le R(n)$. First note that if $r_{n,j}$ denotes the $j$th summand in $R(n)$, then we have $r_{n,j}\ge r_{n+1,j}$ and $r_{n,n-j}\ge r_{n+1,n+1-j}$. Thus it is enough to show 
$$\sum_{i=-1}^1r_{n,\lceil n/2\rceil +i}\ge \sum_{i=-1}^2r_{n+1, \lceil n/2\rceil +i}.$$
By the definition of $m^*_s$, we know that $r_{n,\lceil n/2\rceil}<1$. Since $\binom{n}{\lceil n/2\rceil}=(1/2+o(1))\binom{n+1}{\lceil n/2\rceil}$ we have that the LHS is $(3+o(1))r_{n,\lceil n/2\rceil}$ while the RHS is $(4+o(1))r_{n,\lceil n/2\rceil}/2=(2+o(1))r_{n,\lceil n/2\rceil}$. This finishes the proof of (ii).

\vskip 0.2truecm

Finally, we prove (iii). Clearly, as long as $n \le m^*_s$ for any family $\cG\subseteq 2^{[n]}$ the number of pairs is $(n+1)n!\le (m^*_s+1)n!$. We need to show that $R(n)\le m^*_s+1$ holds for all $n> m^*_s$. As in (ii) the proof of $R(n+1) \le R(n)$ for $n \ge m^*_s$ did not require the assumption on $s$ and $c$, we obtain that $R(n)\le m^*_s+1$ holds for all $n$.
\end{proof}

Our last auxiliary lemma was proved by Griggs, Li and Lu \cite{GLL_dia}.

\begin{lemma} [Griggs, Li, Lu, during the proof of Theorem 2.5 in \cite{GLL_dia}]
\label{lem:nonindlem} Let $s \ge 2$, $r=t=1$ and thus $m_s := \lceil\log_2(s + 2)\rceil$. 

(1) If $s \in [2^{m_s - 1} -1,2^{m_s} - \binom{m_s}{\lceil \frac{m_s}{2}\rceil}-1]$, then if $\cG \subseteq 2^{[n]}$ is a $K_{1,s,1}$-free family of sets, then the number of pairs $(G,\cC)$ with $G \in \cG\cap \cC$ and $\cC \in \bC_n$ is at most $m_sn!$.

(2) If $s \in [2^{m_s} - \binom{m_s}{\lceil \frac{m_s}{2}\rceil},2^{m_s} -2]$, then if $\cG \subseteq 2^{[n]}$ is a $K_{1,s,1}$-free family of sets, then the number of pairs $(G,\cC)$ with $G \in \cG\cap \cC$ and $\cC \in \bC_n$ is at most $(m_s+1-\frac{2^{m_s}-s-1}{\binom{m_s}{\lceil \frac{m_s}{2}\rceil}})n!$. 

\end{lemma}

\section{Proofs}
\label{sec:pr}
In this section we prove our main theorems. Let us start with constructions to see the lower bounds. 
We partition $\binom{[n]}{k}$ into $n$ classes: $\cF_{n,k,i}=\{F\in \binom{[n]}{k}: \sum_{j\in F}j\equiv i \hskip 0.2truecm (\mod n)\}$ and denote the union of the $r$ largest classes by $\binom{[n]}{k}_{r, mod}$. Clearly, $|\binom{[n]}{k}_{r, mod}|\ge \frac{r}{n}\binom{n}{k}$. Furthermore, it has the property that for any distinct $r+1$ sets $F_1,F_2,\dots, F_{r+1}\in \binom{[n]}{k}_{r, mod}$ we have $|\cap_{i=1}^{r+1}F_i|\le k-2$ and $|\cup_{i=1}^{r+1}F_i|\ge k+2$.

\begin{itemize}
\item
For \tref{twopart} consider the family $\cF:=\binom{[n]}{\lceil n/2\rceil -2}_{r-1,mod}\cup\binom{[n]}{\lceil n/2\rceil -1}\cup \binom{[n]}{\lceil n/2\rceil} \cup \binom{[n]}{\lceil n/2\rceil +1}_{s-1,mod}$. Suppose $A_1,A_2,\dots,A_r, B_1,B_2, \dots,B_s\in\cF$ form an induced copy of $K_{r,t}$. Then $\cup_{i=1}^rA_i\subseteq \cap_{j=1}^sB_j$ holds, but by the above property of $\binom{[n]}{k}_{r, mod}$ and the inducedness we have $|\cup_{i=1}^rA_i|\ge \lceil n/2\rceil$ and  $|\cap_{j=1}^sB_j|\le \lceil n/2\rceil-1$ - a contradiction.
\item
For \tref{threepartnonind} let $k$ be the index of the level below the $m_s+f(r,t)$ middle levels, i.e., $k=\lceil\frac{n-m_s-f(r,t)}{2}\rceil-1$. Write $l=k+m_s+f(r,t)+1$ and let us consider the family 
$$\cF:=\binom{[n]}{k}_{(r-2)^+,mod} \cup \bigcup_{i=1}^{m_s+f(r,t)}\binom{[n]}{k+i} \cup \binom{[n]}{l}_{(t-2)^+,mod}.$$ We claim that $\cF$ is $K_{r,s,t}$-free. Assume not and let $A_1,A_2,\dots,A_r, B_1,B_2, \dots,B_s$, $C_1,C_2,\dots,C_t\in\cF$ form a copy of $K_{r,s,t}$. If $r\ge 2$, then $|\cup_{i=1}^rA_i|\ge k+2$ and if $r=1$, then $|A_1|\ge k+1$ (note that if $r=1,2$, then $(r-2)^+=0$ and thus the smallest set size in $\cF$ is $k+1$). Similarly, if $t\ge 2$, then $|\cap_{j=1}^tC_j|\le l-2$ and if $t=1$, then $|C_1|\le l-1$. In any case, $|\cap_{t=1}^tC_j|-|\cup_{i=1}^rA_i| \le m_s-1$ and thus there is no place for $B_1,B_2,\dots,B_s$ - a contradiction.
\item
The construction showing the lower bound of \tref{rest} is a special case of the one for \tref{threepartnonind}.
\item
For \tref{threepartind} (i), (ii) and (iii), let $k$ be the index of the level below the $m^*_s+f(r,t)$ middle levels, i.e., $k=\lceil\frac{n-m^*_s-f(r,t)}{2}\rceil-1$. Write $l=k+m^*_s+f(r,t)+1$ and let us consider the family 
$$\cF:=\binom{[n]}{k}_{r-1,mod} \cup \bigcup_{i=1}^{m^*_s+f(r,t)}\binom{[n]}{k+i} \cup \binom{[n]}{l}_{t-1,mod}.$$
One can see that for any antichains $A_1,A_2,\dots, A_r\in\cF$ and $C_1,C_2,\dots, C_t\in\cF$ we have $|\cap_{i=1}^tC_i|-|\cup_{j=1}^rA_j|\le m^*_s-1$ and thus there is no room for an antichain of size $s$ in between. Note that when $s=4$, then $m^*_s=4$ as $\binom{4}{2}=6\ge 4$, but $\binom{3}{2}=3<4$.
\end{itemize}

\vskip 0.3truecm

Let us now start proving the upper bounds of our results. First of all, from here on every family $\cF\subseteq 2^{[n]}$ contains sets only of size from the interval $[n/2-n^{2/3},n/2+n^{2/3}]$. This leaves all our proofs valid as by Chernoff's inequality $|\{F\subseteq [n]: ||F|-n/2|\ge n^{2/3}\}|\le 2e^{-2n^{1/3}}=o(\frac{1}{n^2}\binom{n}{\lceil n/2\rceil})$. 

As we mentioned in the Introduction, for all proofs we will use the chain partition method. This works in the following way: for a family $\cF\subseteq 2^{[n]}$ suppose we can partition $\bC_n$ into $\bC_{n,1},\bC_{n,2},\dots\bC_{n,l}$ such that for all $1\le i \le l$ the number of pairs $(F,\cC)$ with $F\in \cF\cap\cC$ and $\cC\in \bC_{n,i}$ is at most $b|\bC_{n,i}|$. Then clearly the number of pairs  $(F,\cC)$ with $F\in \cF\cap\cC$ and $\cC\in \bC_{n}$ is at most $b|\bC_{n}|$. Since the number of such pairs is exactly $\sum_{F\in\cF}|F|!(n-|F|)!$ we obtain the LYM-type inequality
\[
\sum_{F\in\cF} \frac{1}{\binom{n}{|F|}}\le b
\]
and thus $|\cF|\le b\binom{n}{\lceil n/2\rceil}$ holds. Therefore, in the proofs below we will end our reasoning whenever we reach a bound on the appropriate partition as mentioned above.

\begin{proof}[Proof of the upper bound in \tref{twopart}] Let $\cF$ be an induced $K_{r,t}$-free family. We can assume that $\cF$ contains an antichain of size at least $r$ as otherwise $\cF$ could contain at most $r-1$ sets of the same size and thus we would obtain $|\cF|\le (r-1)(n+1)$.

We will use the $\min_r$-partition $\{\bC_{A,r}: s^-_{\cF}(A)\ge r\}$.

We claim that the number of pairs $(F,\cC)$ with $F\in \cF\cap \cC$ and $\cC\in \bC_{A,r}$ is at most $(2+\frac{2(r+t-2)}{n}+o(\frac{1}{n}))|\bC_{A,r}|$ for any $A$ with $s^-_{\cF}(A)\ge r$. Note that as $\cF$ is induced $K_{r,t}$-free, for any such $A$ we have $s^+_{\cF}(A)<t$. We distinguish three types of pairs:
\begin{enumerate}
\item
if $A \in \cF$, then there are exactly $|\bC_{A,r}|$ pairs with $F=A$,
\item
any chain in $\bC_{A,r,-}$ can be extended to $(n-|A|)!$ chains in $\bC_{A,r}$, thus by \cref{plus} (i) there are $(1+\frac{2(r-1)}{n}+o(\frac{1}{n}))|\bC_{A,r}|$ pairs with $F \subsetneq A$,
\item
finally, any maximal chain from $A$ to $[n]$ can be extended to $|\bC_{A,r,-}|$ chains in $\bC_{A,r}$, thus \lref{noincomp} (ii) implies that there are $(\frac{2(t-1)}{n}+o(\frac{1}{n}))|\bC_{A,r}|$ pairs with $A \subsetneq F$,
\end{enumerate}
This gives us a total of at most $(2+\frac{2(r+t-2)}{n}+o(\frac{1}{n}))|\bC_{A,r}|$ pairs, which completes the proof.
\end{proof}

\vskip 0.3truecm

Now we turn our attention to complete three level posets.

\begin{proof}[Proof of the upper bound in \tref{threepartnonind}] Let $\cF$ be a $K_{r,s,t}$-free family. We can assume that $\cF$ contains an antichain of size at least $z:=\max\{r,t\}$ as otherwise $\cF$ could contain at most $z-1$ sets of the same size and thus we would obtain $|\cF|\le (z-1)(n+1)$.

We consider the $\min_r-\max_t$-partition of $\bC_n$ and we claim that the number of pairs $(F,\cC)$ with $F\in \cF\cap \cC$ and $\cC \in \bC_S$, $\cC\in \bC_{A,B}$ is at most $b|\bC_S|$, $b|\bC_{A,B}|$, respectively, where $b=(m_s+f(r,t)+\frac{2(r+t-2)}{n}+o(\frac{1}{n}))$ when we prove (1) and $b=(m_s+f(r,t)+1- \frac{2^{m_s}-s+f(r,t)-1}{\binom{m}{\lceil \frac{m_s}{2}\rceil}})$ when we prove (2). 

First consider the ``degenerate" case of $\bC_S$ with $S\in \cS^-$. A chain $\cC \in \bC_{S}$ goes from $\emptyset$ until one of the subsets $S_1,S_2,\dots, S_k$ of $S$ with size $|S|-1$ for which $s^-_{\cF}(S_i)<r$, then $\cC$ must go through $S$, and finally $\cC$ must contain a maximal chain from $S$ to $[n]$. Thus $|\bC_{S}|=k(|S|-1)!(n-|S|)!$. We distinguish two types of pairs to count.
\begin{enumerate}
\item
If $r\ge 2$, then applying \cref{plus} (i) we obtain that there are at most $(1+\frac{2(r-1)}{n}+o(\frac{1}{n}))|\bC_{S}|$ pairs $(F,\cC)$ with $F\subsetneq S$. Together with $\{(S,\cC):\cC \in \bC_S\}$ we have $(2+\frac{2(r-1)}{n}+o(\frac{1}{n}))|\bC_{S}|$ pairs. If $r=1$, then by definition the number of pairs $(F,\cC)$ with $F\subseteq S$ is at most $|\bC_S|$ as for all such pairs we must have $F=S$.
\item
Applying \lref{noincomp} (ii) we obtain that there are at most $(\frac{2(t-1)}{n}+o(\frac{1}{n}))|\bC_{S}|$ pairs $(F,\cC)$ with $S\subsetneq F$.
\end{enumerate}
This gives a total of at most $(2+\frac{2(r+t-2)}{n}+o(\frac{1}{n}))|\bC_{S}|$ pairs.

We now consider the ``more natural" $A \in \cS^+$, $A \subseteq B$ case. As there are sets in the interval $[A,B]$, this time we distinguish three types of pairs:
\begin{enumerate}
\item
If $r=1$, then there is no pair $(F,\cC)$ with $F\subsetneq A$. If $r\ge 2$, then applying \cref{plus} (i) we obtain that there are at most $(1+\frac{2(r-1)}{n}+o(\frac{1}{n}))|\bC_{A,B}|$ pairs $(F,\cC)$ with $F\subsetneq A$.
\item
If $t=1$, then there is no pair $(F,\cC)$ with $B\subsetneq F$. If $t\ge 2$, then applying \cref{plus} (ii) we obtain that there are at most $(1+\frac{2(t-1)}{n}+o(\frac{1}{n}))|\bC_{A,B}|$ pairs $(F,\cC)$ with $B\subsetneq F$.
\item
If $\cF$ is a $K_{r,s,t}$-free family, then $\{F\in \cF: A\subseteq F \subseteq B\}$ is a $K_{1,s-f(r,t),1}$-free family. Indeed, if $f(r,t)= 2$, then $|\{F\in \cF: A\subseteq F \subseteq B\}|\le s$ as these sets together with the sets of the antichain of size $r$ below $A$ and the sets of the antichain of size $t$ above $B$ would form a copy of $K_{r,s,t}$ in $\cF$. If $f(r,t)=1$, say $r=1$, then by the definition of the $\min_1-\max_t$-partition, we have $A \in \cF$ and thus $|\{F\in \cF: A\subsetneq F \subseteq B\}|\le s$, in particular together with $A$ they are $K_{1,s-1,1}$-free. If $f(r,t)=0$, then the $K_{1,s-f(r,t),1}$-free property is the same as the $K_{1,s,1}$-free property which is possessed by $\{F\in \cF: A\subseteq F \subseteq B\}$ as it is a subfamily of $\cF$.

By \lref{nonindlem}, in case (1) of \tref{threepartnonind} the number of pairs $(F,\cC)$ with $A\subseteq F\subseteq B$ is at most $m_s|\bC_{A,B}|$, while in case (2)  of \tref{threepartnonind} the number of pairs $(F,\cC)$ with $A\subseteq F\subseteq B$ is at most $(m_s+1- \frac{2^{m_s}-s+f(r,t)-1}{\binom{m_s}{\lceil m_s/2\rceil}})|\bC_{A,B}|$.
\end{enumerate}
Adding up the number of three types of pairs we obtain that the total number of pairs is not more than $(m_s+f(r,t)+\frac{2(r+t-2)}{n}+o(\frac{1}{n}))|\bC_{A,B}|$ and $(m_s+1+f(r,t)- \frac{2^{m_s}-s+f(r,t)-1}{\binom{m_s}{\lceil m_s/2\rceil}}+\frac{2(r+t-2)}{n}+o(\frac{1}{n}))|\bC_{A,B}|$ in the two respective cases of \tref{threepartnonind}. 
\end{proof}

\vskip 0.3truecm

We continue with the proof of \tref{rest}.

\begin{proof}[Proof of \tref{rest}]
Let $\cF$ be a $K_{r,2,t}$-free family and let us write $r^{++}=\max\{r,2\}, t^{++}=\max\{t,2\}$. We consider the $\min_{r^{++}}-\max_{t^{++}}$-partition of $\bC_n$. Just as in the proof of \tref{threepartnonind}, we obtain that if $S\in \cS^-$, then the number of pairs $(F,\cC)$ with $F\in\cF\cap \cC$ and $\cC\in \bC_{S}$ is at most $(2+O(\frac{1}{n}))|\bC_{S}|$. Note that if $A \subseteq B$, then $|\cF\cap\{G\in 2^{[n]}:A\subseteq G\subseteq B\}|\le 1$ as by definition of the $\min_{r^{++}}-\max_{t^{++}}$-partition two such sets would make $\cF$ contain a copy of $K_{r,2,t}$. 
\begin{itemize}
\item
Applying \cref{plus} (i) we obtain that there are at most $(1+\frac{2(r^{++}-1)}{n}+o(\frac{1}{n}))|\bC_{A,B}|$ pairs $(F,\cC)$ with $F\subsetneq A$.
\item
Applying \cref{plus} (ii) we obtain that there are at most $(1+\frac{2(t^{++}-1)}{n}+o(\frac{1}{n}))|\bC_{A,B}|$ pairs $(F,\cC)$ with $B\subsetneq F$.
\item
By the observation above, the number of pairs $(F,\cC)$ with $A \subseteq F\subseteq B$ is at most $|\bC_{A,B}|$.
\end{itemize}
\end{proof}

\vskip 0.3truecm

\begin{proof}[Proof of \tref{threepartind}]
Throughout the proof we will assume that all $s_i$'s are at least 2. This will be needed for the fact that all $m^*_{s_i}$'s are larger than 1.

First we prove (i), (ii), and (iii). Let $\cF$ be an induced $K_{r,s,t}$-free family. We can assume that $\cF$ contains an antichain of size at least $z:=\max\{r,t\}$ as otherwise $\cF$ could contain at most $z-1$ sets of the same size and thus we would obtain $|\cF|\le (z-1)(n+1)$.
 We again consider the $\min_r-\max_t$-partition of $\bC_n$ and count the number of pairs $(F,\cC)$ with $F\in \cF\cap \cC$ and $\cC\in \bC_n$.

The degenerate case is identical to what we had in the proof of \tref{threepartnonind}, thus we only consider the case when $A \in \cS^+$, $A \subseteq B$. The three types of pairs:
\begin{enumerate}
\item
If $r=1$, then there is no pair $(F,\cC)$ with $F\subsetneq A$. If $r\ge 2$, then applying \cref{plus} (i) we obtain that there are at most $(1+\frac{2(r-1)}{n}+o(\frac{1}{n}))|\bC_{A,B}|$ pairs $(F,\cC)$ with $F\subsetneq A$.
\item
If $t=1$, then there is no pair $(F,\cC)$ with $B\subsetneq F$. If $t\ge 2$, then applying \cref{plus} (ii) we obtain that there are at most $(1+\frac{2(t-1)}{n}+o(\frac{1}{n}))|\bC_{A,B}|$ pairs $(F,\cC)$ with $B\subsetneq F$.
\item
Note that $\{F\in \cF: A\subseteq F\subseteq B\}$ cannot contain an antichain of size $s$ as otherwise $\cF$ would contain an induced copy of $K_{r,s,t}$.
\begin{enumerate}
\item
If $\cF$ is an induced $K_{r,4,t}$-free family, then by \lref{ind4} (i) the number of pairs $(F,\cC)$ with $A\subseteq F\subseteq B$ is at most $4|\bC_{A,B}|$.
\item
If $\cF$ is an induced $K_{r,s,t}$-free family with $s\le c\binom{m^*_s}{\lceil m^*_s/2\rceil}$ and $s$ large enough, then by \lref{ind4} (ii) the number of pairs $(F,\cC)$ with $A\subseteq F\subseteq B$ is at most $m^*_s|\bC_{A,B}|$.
\item
If $\cF$ is an induced $K_{r,s,t}$-free family with $s$ large enough, then by \lref{ind4} (iii) the number of pairs $(F,\cC)$ with $A\subseteq F\subseteq B$ is at most $(m^*_s+1)|\bC_{A,B}|$.
\end{enumerate}
\end{enumerate}
Altogether these bounds yield that the total number of pairs is at most 
\begin{enumerate}
\item
$(4+f(r,t)+\frac{2(r+t-2)}{n}+o(\frac{1}{n}))|\bC_n|$ if $\cF$ is induced $K_{r,4,t}$-free.
\item
$(m^*_s+f(r,t)+\frac{2(r+t-2)}{n}+o(\frac{1}{n}))|\bC_n|$ if $\cF$ is induced $K_{r,s,t}$-free, $s\le c\binom{m^*_s}{\lceil m^*_s/2\rceil}$ and $s$ large enough.
\item
$(m^*_s+1+f(r,t)+\frac{2(r+t-2)}{n}+o(\frac{1}{n}))|\bC_n|$ if $\cF$ is induced $K_{r,s,t}$-free and $s$ large enough.
\end{enumerate}

Now we prove (iv). Let $\cF$ be an induced $K_{r,s_1,s_2,\dots,s_j,t}$-free family.  We can assume that $\cF$ contains an antichain of size at least $z:=\max\{r,t\}$ as otherwise $\cF$ could contain at most $z-1$ sets of the same size and thus we would obtain $|\cF|\le (z-1)(n+1)$. Before proceeding with the formal proof, let us briefly summarize the ideas of the partition of $\bC_n$ that we are going to use. Just as in the case of the $\min_r-\max_t$-partition we try to assign markers $A_0,A_1,\dots,A_j$ to every chain $\cC\in \bC_n$ with the following properties: (a) $A_0$ is the smallest set in $\cC$ with $s^-_{\cF}(A_0)\ge r$ and (b) for every $1\le i\le j$ $A_i$ is the smallest set in $\cC$ \textit{above $A_{i-1}$} such that $[A_{i-1},A_i]$ contains an antichain of size $s_i$. This definition enables us to build the $i$th middle level of $K_{r,s_1,\dots,s_j,t}$ between $A_{i-1}$ and $A_i$ for all $i$ with $1\le i\le j$ and thus we obtain that $s^+_{\cF}(A_j)<t$ must hold. If we were able to define all those markers, then we could apply our lemmas from \sref{lemmas} to bound the number of pairs $(F,\cC)$ with $F\in \cF\cap \cC$, $\cC\in\bC_n$ in the different intervals $[A_i,A_{i+1}]$. Unfortunately, it might happen that not all markers can be defined. However we will index the parts of the partition of $\bC_n$ by chains of length at most $j+1$. Instead of giving formal definitions of the $\bC_{A_0,\dots, A_i}$'s and then verifying that they indeed form a partition of $\bC_n$, we consider an arbitrary maximal chain $\cC\in \bC_n$ and describe the procedure how to define its markers.

\begin{itemize}
\item
If $r=1$, then $A_0$ is the smallest set in $\cF \cap \cC$,
\item
if $r\ge 2$, then $A_0$ is the smallest set in $\cC$ with $s^-_{\cF}(A_0)\ge r$.
\end{itemize}
Note that by the assumption $s^-_{\cF}([n])\ge \max\{r,t\}$ the marker $A_0$ is defined for all chains $\cC\in\bC_n$. Let us now assume that $A_{i-1}$ has been defined for some $1\le i\le j$. If $s^+_{\cF}(A_{i-1})<s_i$, then our procedure is finished and $\cC$ belongs to $\bC_{A_0,A_1,\dots,A_{i-1}}$. If $s^+_{\cF}(A_{i-1})\ge s_i$ holds, then

\begin{itemize}
\item
$A_i$ is the smallest set in $\cC$ such that $[A_{i-1},A_i]$ contains an antichain of size $s_i$.
\end{itemize}
Note that if the procedure does not stop at $A_{i-1}$, then $A_i$ exists as $[n]\in \cC$ and $s_i\le s^+_{\cF}(A_{i-1})$.

Observe that a chain $\cC$ in $\bC_{A_0,\dots,A_i}$ contains all $A_k$'s and for every $0\le k\le i$ it goes through one of the $(|A_k|-1)$-subsets $A^k_1,\dots A^k_{l_k}$ of $A_k$ for which $[A_{k-1},A^k_l]$ does not contain an antichain of size $s_k$ where $A_{-1}=\emptyset$ and $s_{0}=r$. 

We now count the pairs $(F,\cC)$ with $F\in \cF\cap\cC$ and $\cC \in \bC_{A_0,\dots,A_i}$. 
\begin{itemize}
\item
Pairs with $F\subsetneq A_0$. If $r=1$, then there is no such pair by definition of $A_0$, otherwise we can apply \cref{plus} (i) to $A_0$ to obtain that the number of such pairs is at most $(1+\frac{2(r-1)}{n}+o(\frac{1}{n}))|\bC_{A_,\dots, A_j}|$.
\item
Pairs with $A_0 \subseteq F\subsetneq A_{i}$. For any $1\le k \le i$ one can apply \lref{ind4} (ii) to $A_{k-1}$ and all $A^k_1,\dots, A^k_{l_k}$ to obtain that the number of pairs with $F\in [A_{k-1},A^k_l]$ for some $1\le l\le l_k$ is at most $m^*_{s_k}|\bC_{A_0,\dots,A_i}|$.
\item
Pairs with $F\supseteq A_i$. 
\begin{itemize}
\item
If $i<j$, then by definition of how we declared our process finished, we obtain $s^+_{\cF}(A_i)<s_{i+1}$.
Thus we can apply \lref{noincomp} (ii) to obtain that the number of such pairs is at most $(1+\frac{2(s_{i+1}-1)}{n}+o(\frac{1}{n}))|\bC_{A_,\dots, A_i}|$.
\item
If $i=j$ and $t=1$, then by definition of $A_j$ there is no such pair. 
\item
If $i=j$ and $t>1$, then as $\cF$ is induced $K_{r,s_1,\dots,s_j,t}$-free, we obtain that $s^+_{\cF}(A_j)<t$. Thus we can apply \lref{noincomp} (ii) to obtain that the number of such pairs is at most $(1+\frac{2(t-1)}{n}+o(\frac{1}{n}))|\bC_{A_,\dots, A_j}|$.
\end{itemize}
\end{itemize}

Adding up these bounds we obtain that if $i=j$, then the total number of pairs is at most $(f(r,t)+\sum_{k=1}^jm^*_{s_k}+O(\frac{1}{n}))|\bC_{A_0,\dots,A_j}|$. If $i<j$ holds the upper bound we obtain is $(f(r,t)+1+\sum_{k=1}^im^*_{s_k}+O(\frac{1}{n}))|\bC_{A_0,\dots,A_j}|$. But since $s_j>1$ holds, we have $(f(r,t)+1+\sum_{k=1}^im^*_{s_k}+O(\frac{1}{n}))|\bC_{A_0,\dots,A_j}|\le (f(r,t)+\sum_{k=1}^jm^*_{s_k}+O(\frac{1}{n}))|\bC_{A_0,\dots,A_j}|$.
\end{proof}

\vskip 0.3truecm

\noindent\textbf{Acknowledgment.} I would like to thank an anonymous referee for his many helpful remarks to improve the presentation of the paper.

\end{document}